\documentclass[10pt]{article}       
\title{The Stanely-F\'eray-\'Sniady  formula for the generalized characters of the Symmetric Group}

\author{Fabio Scarabotti}

\usepackage{latexsym,amsbsy,amsmath,amscd,amssymb}  

\usepackage{amsthm}       
\usepackage{amssymb,amsmath,latexsym}
 \newtheorem{definition}{Definition} [section]       
        
 \newtheorem{example}[definition]{Example}

 \newtheorem{proposition}[definition]{Proposition}       
 \newtheorem{theorem}[definition]{Theorem}       
        
  \newtheorem{lemma}[definition]{Lemma}

\begin{document}

\maketitle

\begin{abstract}  
We show that the explicit formula of Stanley-F\'eray-\'Sniady  for the characters of the symmetric group have a natural extension to the generalized characters. These are the spherical functions of the unbalanced Gel'fand pair $(S_n\times S_{n-1},\text{diag}S_{n-1})$.
\footnote{{\it AMS 2002 Math. Subj. Class.}Primary: 20C30. Secondary: 20C15, 43A90

\it Keywords: Symmetric group; Gel'fand pair, spherical function, generalized character}
\end{abstract}

\section{Introduction}
Recently, two different explicit formulas have been found for the characters of the symmetric group, the Stanley-F\'eray formula, conjectured by R. P. Stanley \cite{Stanley} and proved by V. F\'eray \cite{Feray}, and the formula of M. Lassalle \cite{Lassalle} (se \cite{CST2} for an account). Actually, these are formulas for {\em spherical functions} rather than characters. Indeed, these formulas give the normalized characters obtained by dividing each of them by the dimension of the corresponding representation. These are the spherical functions of the Gel'fand pair $(S_n\times S_n,\text{diag}S_n)$, where $\text{diag}S_n=\{(\pi,\pi):\pi\in S_n\}$. In \cite{Strahov} E. Strahov showed that some of the classical results for the characters of the symmetric group may be extended to the spherical functions of the unbalanced Gel'fand pair $(S_n\times S_{n-1},\text{diag}S_{n-1})$, where $\text{diag}S_{n-1}=\{(\pi,\pi):\pi\in S_{n-1}\}$. This amounts to consider the algebra of all $S_{n-1}$-conjugacy invariant functions on $S_n$ rather than the $S_n$-conjugacy invariant functions. It is a natural problem to extend a result for the normalized characters to the generalized characters of the symmetric group. In the present paper we show that the Stanley-F\'eray formula, in the form proved by F\'eray and P. \'Sniady  in \cite{FeraySniady}, may be naturally extended to the generalized characters.

\section{Preliminaries}

We recall some basic facts on unbalanced Gel'fand pairs. We refer to \cite{CST1,CST2,Strahov,Travis} for more details and proofs (but we follows the notation in our joint monographs with T. Ceccherini-Silberstein and F. Tolli). If $X$ is a finite set, we denote by $L(X)$ the space of all complex valued functions defined on $X$. Let $G$ be a finite group. We say that $H\leq G$ is a {\em multiplicity free subgroup} of $G$ when $\text{Res}^G_H\sigma$ is a multiplicity free representation of $H$ for every irreducible representation $\sigma$ of $G$. We recall that the action of $G\times H$ on $G\equiv \frac{G\times H}{\text{diag}H}$ is $(g,h)\cdot g_0=gg_0h^{-1}$. The subgroup $H$ is multiplicity free if and only if $(G\times H,\text{diag}H)$ is a Gel'fand pair, if and only if the algebra of $H$-conjugacy invariant functions on $G$ is commutative. Let $\widehat{G}$ (resp. $\widehat{H}$) be a complete set of pairwise inequivalent (unitary) irreducible representation of $G$ (resp. $H$). For $\sigma\in\widehat{G}$ we denote by $\sigma'$ the adjoint of $\sigma$. If $\rho\in\widehat{H}$ and $\sigma\in\widehat{G}$, we write $\rho\leq\text{Res}^G_H\sigma$ to denote that $\rho$ is contained in $\text{Res}^G_H\sigma$; $\sigma \boxtimes \rho$ denotes the tensor product of $\sigma$ and $\rho$; $\chi^\sigma$ and $\chi^\rho$ are the characters of $\sigma$ and $\rho$ (they are not normalized: $\chi^\rho(1_G)$ is equal to the dimension $d_\rho$ of $\rho$). If $H$ is multiplicity free, the decomposition of permutation representation $\eta$ of $G\times H$ on $L(G)$ is the following:

\begin{equation}\label{abstractdec}
\eta\cong\bigoplus_{\sigma\in \widehat{G}}\bigoplus_{\substack{\rho\in\widehat{H}:\\ \rho\leq\text{Res}^G_H\sigma'}}(\sigma \boxtimes \rho).
\end{equation}

\noindent
In particular, for $H=G$ the $G\times G$-irreducible representation $\sigma\boxtimes \sigma'$ coincides with the $\sigma$-isotypic component in $L(G)$, that is the subspace of $L(G)$ spanned by the matrix coefficients of $\sigma$. The spherical function associated to $\sigma \boxtimes \rho$ has the following expression:

\[
\phi_{\sigma,\rho}(g)=\frac{1}{\lvert H\rvert}\sum_{h\in H}\overline{\chi^\sigma(gh)}\overline{\chi^\rho(h)}.
\]

\noindent
Following \cite{Strahov}, we call $\phi_{\sigma,\rho}$ a {\em generalized character} of $G$.

\begin{proposition}\label{propgenchar}
Suppose that $H$ is a multiplicity free subgroup of $G$. With the notation above we have:
\begin{enumerate}
\item
$\phi_{\sigma,\rho}(h)=\frac{1}{d_\rho}\chi^\rho(h)$ for all $h\in H$;

\item
if $\psi\in L(G)$ is $H$-conjugacy invariant, it belongs to the $\sigma$-isotypic component of $L(G)$ and $\psi(h)=\frac{1}{d_\rho}\chi^\rho(h)$ for all $h\in H$ then $\psi=\phi_{\sigma,\rho}$.
\end{enumerate}

\end{proposition}
\begin{proof}

Suppose that $\text{Res}^G_H\sigma'=\bigoplus_{i=1}^m\rho_i$, with $\rho_1,\rho_2,\dotsc,\rho_m\in\widehat{H}$ (pairwise inequivalent) and $\rho_1=\rho$.

\begin{enumerate}

\item
For every $h\in H$ we have:

\[
\phi_{\sigma,\rho}(h)=\frac{1}{\lvert H\rvert}\sum_{t\in H}\overline{\chi^\sigma(th)}\chi^\rho(t^{-1})=\frac{1}{\lvert H\rvert}\left[\left(\sum_{i=1}^m\chi^{\rho_i}\right)*\chi^\rho\right](h)=\frac{1}{d_\rho}\chi^\rho(h).
\]

\item
Since $\psi$ is $H$-conjugacy invariant and belongs to the $\sigma$-isotypic component of $L(G)$ then $\psi=\sum_{i=1}^mc_i\phi_{\sigma,\rho_i}$, for suitable complex constants $c_1,c_2,\dotsc,c_m$. Therefore

\[
\frac{1}{d_\rho}\chi^\rho(h)=\psi(h)=\sum_{i=1}^mc_i\phi_{\sigma,\rho_i}(h)=\sum_{i=1}^m\frac{c_i}{d_{\rho_i}}\chi^{\rho_i}(h)\qquad \text{for all}\qquad h\in H
\]

\noindent
implies that $c_i=\delta_{i,1}$, that is $\psi=\phi_{\sigma,\rho}$.
\end{enumerate}
\end{proof}

\section{Brender's formula}
In this section we give a short proof of the main result in \cite{Brender}. It is a formula for the generalized character of the symmetric group analogous to (8) in \cite{FeraySniady}.
Let $S_n$ be the symmetric group of degree $n$. We think of it as the group of all permutations of the set $\{1,2,\dotsc,n\}$. We denote by $\widetilde{S}_{n-1}$ the stabilizer of $1$ in $S_n$. Then \cite{Brender,CST1,CST2,Strahov} $\widetilde{S}_{n-1}$ is a multiplicity free subgroup of $S_n$. If $\lambda\vdash n$ then $S^\lambda$ is the corresponding irreducible $S_n$-representation. We identify $\lambda\vdash n$ with its Young frame; if $\lambda\vdash n$ and $\mu\vdash n-1$ we write $\lambda\rightarrow \mu$ to denote that $\mu$ may obtained from $\lambda$ by removing one box (we denote by $\lambda\setminus\mu$ this box); note that Strahov draws the arrow in the opposite sense. Then the branching rule for the symmetric group may be written in the form:  $\text{Res}^{S_n}_{S_{n-1}}S^\lambda=\bigoplus_{\substack{\mu\vdash n-1:\\ \lambda\rightarrow \mu}}S^\mu$. Therefore, in the present setting, \eqref{abstractdec} is:

\[
L(S_n)=\bigoplus_{\lambda\vdash n}\bigoplus_{\substack{\mu\vdash n-1:\\ \lambda\rightarrow \mu}}\left(S^\lambda\boxtimes S^\mu\right).
\]

\noindent
We denote by $\phi_{\lambda,\mu}$ the generalized character associated to $S^\lambda\boxtimes S^\mu$.\\

We will use the following notation: a function $f\in L(S_n)$ will be identified with the formal sum $\sum_{\pi\in S_n}f(\pi)\pi$. If $t$ is a $\lambda$-tableau (an injective filling of the Young frame of $\lambda$ with the numbers $\{1,2,\dotsc,n\}$), we denote by $R_t$ (resp. $C_t$) the row (resp. the column) stabilizer of $t$. It is well known that the element

\[
E_t=\sum_{\gamma\in C_t}\sum_{\sigma\in R_t}\text{sign}(\gamma)\gamma\sigma,
\] 

\noindent
belongs to the $\lambda$-isotypic component of $L(S_n)$ \cite{FultonHarris,Simon}: it is a multiple of an idempotent that projects onto a minimal left ideal of $L(S_n)$ isomorphic to $S^\lambda$ (see also Exercise 10.6.7 in \cite{CST1} for a less standard proof). Denote by $\chi^\lambda$ the character of $S^\lambda$ and by $d_\lambda$ the dimension of $S_n$. We have:

\begin{equation}\label{charformula}
\chi^\lambda=\frac{d_\lambda}{n!}\sum_{\pi\in S_n}\pi E_t\pi^{-1}.
\end{equation}

\noindent
The proof is immediate: $f\longmapsto \frac{1}{n!}\sum_{\pi\in S_n}\pi f\pi^{-1}$ is the orthogonal projection from $L(S_n)$ onto the subalgebra of $S_n$-conjugacy invariant functions and the value of $E_t$ on $1_{S_n}$ is 1. See again Exercise 10.6.7 in \cite{CST1} or (VI.6.1) in \cite{Simon} (where $\frac{n!}{d_\mathcal{F}}$ must be replaced by $\frac{d_\mathcal{F}}{n!}$) or (8) in \cite{FeraySniady}.

\begin{proposition}\cite{}\label{Travisprop}
If $\lambda\vdash n$, $\mu\vdash n-1$, $\lambda\rightarrow\mu$ and $t$ is a $\lambda$-tableau with 1 in the box $\lambda\setminus\mu$ then:

\begin{equation}\label{Travisformula}
\phi_{\lambda,\mu}=\frac{1}{(n-1)!}\sum_{\pi\in \widetilde{S}_{n-1}}\pi E_t\pi^{-1}.
\end{equation}
\end{proposition} 
\begin{proof}
The right hand side of \eqref{Travisformula} is $\widetilde{S}_{n-1}$-conjugacy invariant and belongs to the $S^\lambda$-isotypic component of $L(S_n)$. Moreover, following Travis we may write $E_t=E_{t'}+\xi$, where $t'$ is the $\mu$-tableau obtained removing the box containing 1 and

\[
\xi=\sum_{\substack{\gamma\in C_t, \sigma\in R_t\\ \text{but}\;\gamma\notin\widetilde{S}_{n-1}\;\text{or}\;\sigma\notin\widetilde{S}_{n-1}}}\text{sign}(\gamma)\gamma\sigma.
\]

\noindent
Both $\xi$ and $\frac{1}{(n-1)!}\sum_{\pi\in \widetilde{S}_{n-1}}\pi \xi\pi^{-1}$ do not contain elements of $\widetilde{S}_{n-1}$: if 
$\gamma\in C_t$, $\sigma\in R_t$ but $\gamma\notin\widetilde{S}_{n-1}$ or $\sigma\notin\widetilde{S}_{n-1}$ then $\gamma\sigma\notin\widetilde{S}_{n-1}$. Therefore from \eqref{charformula} applied to $\widetilde{S}_{n-1}$ we get:

\[
\frac{1}{(n-1)!}\sum_{\pi\in \widetilde{S}_{n-1}}\pi E_t\pi^{-1}=\frac{1}{d_\mu}\chi^\mu+\frac{1}{(n-1)!}\sum_{\pi\in \widetilde{S}_{n-1}}\pi \xi\pi^{-1}.
\]

\noindent
Now invoking 2. in Proposition \ref{propgenchar} we get the desired result.
\end{proof}

\section{The Stanley-F\'eray-\'Sniady  formula for the generalized characters}

Let $\lambda$, $\mu$ and $t$ be as in Proposition \ref{Travisprop}.  For $\gamma,\sigma\in S_n$ we set:

\[
\begin{split}
\widetilde{N}^{\lambda,\mu}=&\text{ the number of } \pi\in \widetilde{S}_{n-1}\text{ such that: each cycle of } \gamma \text{ is contained in a column of }\pi t \text{ and}\\ 
&\text{ each cycle of } \sigma \text{ is contained in a row of } \pi t.
\end{split}
\]

\noindent
As in \cite{FeraySniady}, if $\square$ is a box of $\lambda$ we denote by $r(\square)$ and $c(\square)$ respectively the row and the column to which $\square$ belongs. Note also that in our notation, $S_{n-1}$ is the stabilizer of $n$; more generally, $S_l\leq S_n$, $1\leq l\leq n$, is the symmetric group on $\{1,2,\dotsc,l\}$ and we will need to consider elements $\pi$ in $S_l$ but not in $\widetilde{S}_{n-1}$, that is permutations of $\{1,2,\dotsc,l\}$ that do not fix 1. Indeed, 1. in Proposition \ref{propgenchar} tells us that the value of a generalized character $\phi_{\lambda,\mu}$ on an element $\pi\in \widetilde{S}_{n-1}$ is given by the formula for the classical characters. For $\gamma,\sigma\in S_l$, with $2\leq l\leq n$, we set

\[
\begin{split}
\widetilde{N}^{\lambda,\mu}_{S_l}(\gamma,\sigma) =& \text{ the number of one-to-one maps } f:\{1,2,\dotsc,l\}\longrightarrow \lambda \text{ such that: } f(1)=\lambda\setminus\mu, \\
&\quad c\circ f \text{ is constant on each cycle of } \gamma  \text{ and }r\circ f \text{ is constant on each cycle of } \sigma,
\end{split}
\]

\noindent
and (removing the injectivity)

\[
\begin{split}
\widehat{N}^{\lambda,\mu}(\gamma,\sigma) =& \text{ the number of functions } f:\{1,2,\dotsc,l\}\longrightarrow \lambda \text{ such that: } f(1)=\lambda\setminus\mu, \\
&\quad c\circ f \text{ is constant on each cycle of } \gamma  \text{ and }r\circ f \text{ is constant on each cycle of } \sigma.
\end{split}
\]

\noindent
Note that if $\gamma,\sigma\in S_l$ then 

\begin{equation}\label{Ntilde}
\widetilde{N}^{\lambda,\mu}(\gamma,\sigma)=(n-l)!\widetilde{N}^{\lambda,\mu}_{S_l}(\gamma,\sigma).
\end{equation}

\noindent
Indeed, when we compute $\widetilde{N}^{\lambda,\mu}(\gamma,\sigma)$ we need to determine the positions of $1,2,\dotsc,l$ in $\pi t$, while the positions of $l+1,l+2,\dotsc,n$ may be choosen arbitrarily. We recall that $(x)_k=x(x-1)\dotsb (x-l+1)$.

\begin{lemma}\label{lemmapreliminare}
If $\theta\in S_l$ then

\[
\phi_{\lambda,\mu}(\theta)=\frac{1}{(n-1)_{l-1}}\sum_{\substack{\gamma,\sigma\in S_l:\\ \gamma\sigma=\theta}}\text{\rm sign}(\gamma)\widehat{N}^{\lambda,\mu}(\gamma,\sigma).
\]

\end{lemma}
\begin{proof}
We may rewrite \eqref{Travisformula} in the form

\[
\phi_{\lambda,\mu}=\frac{1}{(n-1)!}\sum_{\pi\in \widetilde{S}_{n-1}}\sum_{\gamma\in C_{\pi t}}\sum_{\sigma\in R_{\pi t}}\text{sign}(\gamma)\gamma\sigma.
\]

\noindent
Since 

\[
\begin{split}
&\gamma\in C_{\pi t}\Longleftrightarrow \text{ each cycle of }\gamma \text{ is contained in a column of }\pi t\\
&\sigma\in R_{\pi t}\Longleftrightarrow \text{ each cycle of }\sigma \text{ is contained in a row of }\pi t
\end{split}
\]

\noindent
we deduce that

\[
\phi_{\lambda,\mu}(\theta)=\frac{1}{(n-1)!}\sum_{\substack{\gamma,\sigma\in S_n:\\ \gamma\sigma=\theta}}\text{sign}(\gamma)\widetilde{N}^{\lambda,\mu}(\gamma,\sigma)
\]

\noindent
Suppose that $\theta(i)=i$, $\sigma(i)=j$ and $\gamma(j)=i$, with $i\neq j$. If $i,j$ are contained in a row of $\pi t$ thay cannot be contained in a column of $\pi t$ and viceversa. Therefore if $\widetilde{N}^{\lambda,\mu}(\gamma,\sigma)\neq0$ and $\gamma\sigma=\theta$ then supp$(\gamma)$, supp$(\sigma)\subseteq$ supp$(\theta)$.
In particular, if $\theta\in S_l$, the sum may be restricted to $\gamma,\sigma\in S_l$; keeping into account \eqref{Ntilde} we get

\[
\phi_{\lambda,\mu}(\theta)=\frac{1}{(n-1)_{l-1}}\sum_{\substack{\gamma,\sigma\in S_l:\\ \gamma\sigma=\theta}}\text{sign}(\gamma)\widetilde{N}^{\lambda,\mu}_{S_l}(\gamma,\sigma).
\]

\noindent
One can end the proof using the identity

\[
\sum_{\substack{\gamma,\sigma\in S_l:\\ \gamma\sigma=\theta}}\text{sign}(\gamma)\widetilde{N}^{\lambda,\mu}_{S_l}(\gamma,\sigma)=\sum_{\substack{\gamma,\sigma\in S_l:\\ \gamma\sigma=\theta}}\text{sign}(\gamma)\widehat{N}^{\lambda,\mu}(\gamma,\sigma),
\]

\noindent
which has the same proof of (10) in \cite{FeraySniady}.

\end{proof}

We denote by $C(\pi)$ the set of cycles of a permutation $\pi$. A coloring of the cycles of $\gamma,\sigma$ is a function $h:C(\gamma)\sqcup C(\sigma)\longrightarrow \mathbb{N}$. We set:

\[
\begin{split}
N^{\lambda,\mu}(\gamma,\sigma)=&\text{ the number of colorings } h \text{ of the cycles of } \gamma \text{ and } \sigma \text{ such that:}\\
&\quad -\text{the color of each cycle of } \gamma \text{ is a column of }\lambda;\\
&\quad -\text{the color of each cycle of } \sigma \text{ is a row of }\lambda;\\
&\quad -\text{the color of the cycle of }\gamma\text{ containing 1 is } c(\lambda\setminus\mu);\\
&\quad -\text{the color of the cycle of }\sigma\text{ containing 1 is } r(\lambda\setminus\mu);\\
&\quad -\text{if } c_1\in C(\gamma), c_2\in C(\sigma) \text{ and }c_1\cap c_2\neq\emptyset\text{ then } (h(c_1),h(c_2))\\
&\qquad\text{ are the coordinates of a box in }\lambda.
\end{split}
\]

\noindent
Now we can enunciate and prove the analogous of Theorem 2 in \cite{FeraySniady} for the generalized characters.

\begin{theorem}
If $\theta\in S_l$ then

\[
\phi_{\lambda,\mu}(\theta)=\frac{1}{(n-1)_{l-1}}\sum_{\substack{\gamma,\sigma\in S_l:\\ \gamma\sigma=\theta}}\text{\rm sign}(\gamma)N^{\lambda,\mu}(\gamma,\sigma).
\]
\end{theorem} 
\begin{proof}
It follows from Lemma \ref{lemmapreliminare} and the identity: $N^{\lambda,\mu}(\gamma,\sigma)=\widehat{N}^{\lambda,\mu}(\gamma,\sigma)$. This may be proved by means of 
the following natural bijection $h\longmapsto f$ between the colorings $h$ counted by $N^{\lambda,\mu}(\gamma,\sigma)$ and the functions $f$ counted by
$\widehat{N}^{\lambda,\mu}(\gamma,\sigma)$:

\[
f(m)=(h(c_1),h(c_2))\qquad\text{ if }\qquad c_1\in C(\gamma), c_2\in C(\sigma) \text{ and } m\in c_1\cap c_2.
\]

\end{proof}

\begin{example}
{\rm
Suppose that $r(\lambda\setminus\mu)=i$ and $c(\lambda\setminus\mu)=j$. If $\lambda=(\lambda_1,\lambda_2,\dotsc,\lambda_k)$ and the conjugate partition is $\lambda'=(\lambda_1',\lambda'_2,\dotsc,\lambda'_h)$ then $j=\lambda_i$ and $i=\lambda'_j$. Moreover, 

\[
N^{\lambda,\mu}((1)(2),(12))=\lambda_i\qquad\qquad\text{and}\qquad\qquad N^{\lambda,\mu}((12),(1)(2))=\lambda'_j.
\]

\noindent
Therefore

\[
\phi_{\lambda,\mu}((12))=\frac{\lambda_i-\lambda'_j}{n-1}.
\]

\noindent
This formula, in a slightly different form and by means of a completely different method, was found by P. Diaconis in \cite{Diaconis}, (5.10).   

}
\end{example}

\qquad\\
\qquad\\
\noindent
FABIO SCARABOTTI, Dipartimento SBAI-Sezione Matematica, Universit\`a degli Studi di Roma ``La Sapienza'', via A. Scarpa 8, 00161 Roma (Italy)\\
{\it e-mail:} {\tt scarabot@dmmm.uniroma1.it}\\

\end{document}